\documentclass[11pt]{article}
\usepackage{amssymb,amsmath,latexsym,amscd}
\newtheorem{theorem}[equation]{Theorem}
\newtheorem{corollary}[equation]{Corollary}
\newtheorem{lemma}[equation]{Lemma}
\newtheorem{proposition}[equation]{Proposition}

\newtheorem{remark}[equation]{Remark}
\numberwithin{equation}{section}

\newcommand{\Var}{\hbox{Var}\,}

\newcommand{\ot}{\otimes}

\newcommand{\Z}{\mathbb{Z}}

\newcommand{\fg}{{\mathfrak g}}
\newcommand{\fh}{{\mathfrak h}}

\newcommand{\Max}{\hbox{Max}}

\newcommand{\ga}{\alpha}

\newcommand{\gs}{\sigma}
\newcommand{\eps}{\epsilon}
\newcommand{\proof}{{\bf Proof\ \ }}
\newcommand{\qed}{\hfill $\Box$}

\newcommand{\gl}{\lambda}

\newcommand{\End}{\hbox{End}}

\newcommand{\cL}{{\mathcal L}}

\newcommand{\Aut}{{\rm Aut}}

\newcommand{\Gal}{{\mathcal Gal}}

\newcommand{\kalg}{{k\hbox{-alg}}}

\newcommand\Spec{\text{\rm Spec}\,}

\newcommand{\cI}{\mathcal{I}}
\newcommand{\cM}{\mathcal{M}}
\newcommand{\ev}{\mathrm{ev}}

\title{Maximal Ideals and Representations of Twisted Forms of Algebras}

\author{Michael Lau$\hbox{}^{1 *}$ and
Arturo Pianzola$\hbox{\,}^{2,3}$\thanks{Funding from the Natural Sciences and
Engineering Research Council of Canada is gratefully acknowledged.} \vspace{0.3cm}
\\$\hbox{\ \,}^1${\small Universit\'e Laval,
D\'epartement de math\'ematiques et de statistique},\\ {\small Qu\'ebec, QC, Canada G1V 0A6}\\ {\small Email:\ Michael.Lau@mat.ulaval.ca}\vspace{0.1cm}\\
$\hbox{\ \,}^2${\small University of Alberta, Department of
Mathematical and Statistical Sciences,}\\{\small Edmonton, AB,
Canada T6G 2G1}\\{\small Email:\ a.pianzola@math.ualberta.ca}\\
$\hbox{\ \,}^3${\small Centro de Altos Estudios en Ciencias Exactas,} \\
{\small Avenida de Mayo 866, (1084) Buenos Aires, Argentina}}
\date{}

\begin{document}
\maketitle

\begin{small}
\noindent {\bf Abstract.}
Given a central simple algebra $\fg$ and a Galois extension of base rings $S/R$, we show that the maximal ideals of twisted $S/R$-forms of the algebra of currents $\fg(R)$ are in natural bijection with the maximal ideals of $R$.  When $\fg$ is a Lie algebra, we use this to give a complete classification of the finite-dimensional simple modules over twisted forms of $\fg(R)$.

\bigskip


\bigskip



\end{small}
\maketitle

\section{Introduction}
Let $S/R$ be a (finite) Galois extension of commutative, associative, and unital algebras over a field $k$, and let $\fg$ be a finite-dimensional central simple $k$-algebra.  Let $\cL$ be an  $S/R$-form of $\fg \otimes_k R$, that is, an $R$-algebra $\cL$ such that
\begin{equation}\label{form}
 \cL \otimes_R S \simeq \fg \otimes_k S
 \end{equation}
as algebras over $S$.

In this paper we accomplish two tasks:
\medskip

(1) We establish a natural correspondence between the maximal ideals of $\cL$ and those of the base ring $R.$
\medskip

(2) If $\fg$ is a Lie algebra, $k$ is algebraically closed of characteristic $0$, and $R$ is of finite type, we describe all the finite-dimensional irreducible modules of $\cL$ and classify them up to isomorphism.

\medskip

In what follows, we will denote $\fg \otimes_k S$ as $\fg(S)$. Recall that if $\Gamma$ is the Galois group of $S/R$, then there is a natural correspondence between the set of isomorphism classes of $S/R$-forms of $\fg(R)=\fg\ot_k R$ and the pointed set of non-abelian Galois cohomology $\hbox{H}^1\big(\Gamma,\hbox{Aut}_{S-alg}\, \fg(S)\big).$  See \cite{KnusOjanguren}, for example.
\medskip

For example, consider the multiloop algebra $\cL(\fg,\gs)$, where $\fg$ is a finite-dimensional Lie algebra over an algebraically closed field $k$ and $\gs$ is an $N$-tuple of commuting automorphisms
$$\gs_1,\ldots,\gs_N:\ \fg\rightarrow\fg$$ 
of finite orders $m_1,\ldots,m_N$, respectively.  This is a $\Z^N$-graded Lie subalgebra of the Lie algebra $\fg(S)$, where $S=k[t_1^{\pm 1},\ldots,t_N^{\pm 1}]$:
$$\cL(\fg;\gs)=\bigoplus_{j\in\Z^N}\fg_j\ot t_1^{j_1}t_2^{j_2}\cdots t_N^{j_N},$$
where $\fg_j=\{x\in\fg\ |\ \gs_i(x)=\xi_i^{j_i}x\hbox{\ for all\ }i\}$, for fixed primitive $m_i$th roots of unity $\xi_i\in k$.  Then $\cL(\fg,\gs)$ is an $S/R$-form of $\fg(R)$, where $R=k[t_1^{\pm m_1},\ldots,t_N^{\pm m_N}]$.  The Galois group $\Gamma$ of $S/R$ is $\Z_{m_1}\times\cdots\times\Z_{m_N}$, and the corresponding (constant) $1$-cocycle in $\hbox{H}^1\big(\Gamma,\hbox{Aut}_{S-alg}\, \fg(S)\big)$ is the group homomorphism taking a fixed generator $\ga_i$ of $\Z_{m_i}$ to $\gs_i^{-1}\otimes 1$.  Such algebras play an important role in affine Kac-Moody, toroidal, and extended affine Lie theory.\footnote{For simplicity of notation, we use integral powers of the variables $t_i$, though fractional exponents are sometimes used to work with the absolute Galois group of the base ring $R$ or with twisted modules for vertex algebras.} 

\medskip
We open the paper with a detailed investigation of the maximal ideals of twisted forms $\cL$.\footnote{Throughout this paper, all ideals are assumed to be two-sided unless explicit mention to the contrary.}  Given any ideal $\cI$ of the $R$-algebra $\cL$, we show that there is a unique $\Gamma$-stable ideal $J(\cI)\subseteq S$ for which $\cI\ot_R S$ maps to $\fg\ot_k J(\cI)$ under the isomorphism $\cL\ot_R S\rightarrow\fg\ot_k S$.  As all maximal ideals $\cI$ of the $k$-algebra $\cL$  are $R$-stable, this produces a bijection $\psi:\ \cI\mapsto J(\cI)\cap R$ between maximal ideals of the $k$-algebra $\cL$ and the set $\Max(R)$ of maximal ideals of $R$.  Explicitly, $\psi^{-1}:\ I\mapsto I\cL$ for maximal ideals $I\subseteq R$.

To have access to the attractive results of classical representation theory, we then assume that $\fg$ is a finite-dimensional simple Lie algebra and $R$ is of finite type over an algebraically closed field $k$ of characteristic $0$.  The classification of finite-dimensional simple $\cL$-modules $V$ proceeds by observing that the kernel of the representation $\phi:\ \cL\rightarrow \End{_k}(V)$ is an intersection of a finite collection of distinct maximal ideals $\cI_1,\ldots,\cI_n\subseteq\cL$.  Given any maximal ideals $M_1,\ldots, M_n\in \Max(S)$ lying over the maximal ideals $\psi(\cI_1),\ldots,\psi(\cI_n)\in \Max(R)$, respectively, we obtain evaluation maps
$$\ev_M:\ \cL\hookrightarrow\fg\ot_k S\rightarrow\left(\fg\ot_k S/M_1\right)\oplus\cdots\oplus\left(\fg\ot_k S/M_n\right)\simeq \fg^{\oplus n}.$$
We then use properties of forms to show that $\ev_M$ is surjective and descends to an isomorphism $\ev_M:\ \cL/\ker\phi\stackrel{\simeq}{\rightarrow}\fg^{\oplus n}.$  The finite-dimensional simple $\cL$-modules $V$ are thus pullbacks of tensor products of $\fg$-modules along $\ev_M$:
$$V\simeq V(\gl,M)=V_{\gl_1}(M_1)\ot_k\cdots\ot_k V_{\gl_n}(
M_n),$$
for some nonzero dominant integral highest weights $\gl_1,\ldots,\gl_n$ of $\fg$ (relative to a triangular decomposition $\fg=\mathfrak{n}_-\oplus\mathfrak{h}\oplus\mathfrak{n}_+$) and maximal ideals $M_1,\ldots,M_n\in\Max(S)$, where $V_{\gl_i}(M_i)$ is the simple $\fg$-module of highest weight $\gl_i$, viewed as an $\cL$-module via the composition of maps
$$\cL\stackrel{\ev_{M_i}}{\rightarrow}\fg\ot_k S/M_i\simeq\fg\rightarrow\End (V_{\gl_i}).$$

Two such representations $V(\gl,M)=V_{\gl_1}(M_1)\ot_k\cdots\ot_k V_{\gl_m}(M_m)$ and $V(\mu,N)=V_{\mu_1}(N_1)\ot_k\cdots\ot_k V_{\mu_n}(N_n)$ are isomorphic $(\cL/\ker\phi)$-modules, and thus isomorphic $\cL$-modules, if and only if their highest weights are equal, relative to the induced triangular decomposition
$$\cL/\ker\phi=\ev_M^{-1}(\mathfrak{n}_-^{\oplus n})\oplus\ev_M^{-1}(\mathfrak{h}^{\oplus n})\oplus\ev_M^{-1}(\mathfrak{n}_+^{\oplus n}).$$
The cohomological interpretation of forms leads to an action of the group $\Gamma$ on $P_+\times\Max(S)$, for which $V(\gl,M)\simeq V(\mu,N)$ if and only if $m=n$ and
$$(\gl_i,M_i)= {^{\gamma_i}}(\mu_i,N_i)$$
for some $\gamma_1,\ldots,\gamma_n\in\Gamma$.  This classification (Proposition \ref{mainprop}) is then described in terms of $\Gamma$-invariant functions from the maximal spectrum $\Max(S)$ to the set $P_+$ of dominant integral weights.  This gives a constructive description (Theorem \ref{correspondence}) of the moduli space of finite-dimensional simple $\cL$-modules in terms of finitely supported $\Gamma$-invariant functions $\Max(S) \rightarrow P_+$.

One of our main motivations in the present paper was to generalize and provide more intuitive proofs of previous work on (twisted) loop and multiloop algebras.  See \cite{multiloop} or \cite{senesi} for a summary of past work on this problem.  However, the interpretation of isomorphism classes as spaces of $\Gamma$-equivariant maps used in past work does not generalize to our context of twisted forms.  Instead, the $\Gamma$-equivariant functions had to be reinterpreted as {\em $\Gamma$-invariant} functions $\Max(S)\rightarrow P_+$.  This turned out to be the correct perspective to include cases where there is no natural action of $\Gamma$ on the space $P_+^\times$ of nonzero dominant integral weights.  More significantly, with new proofs, we have eliminated all dependence on the $\Z^N$-grading of $\cL(\fg,\gs)$, a point that was crucial in the arguments of \cite{multiloop}.  This lets us apply our work to non-graded contexts, including a classification of modules for the mysterious Margaux algebras explained in Section \ref{ot}.


Perhaps the most striking feature of the present work is its nearly complete independence from the particular $S/R$-form under consideration.  The maximal ideals of {\em any} $S/R$-form $\cL$ of $\fg(R)$ are in bijection with $\Max(R)$, and the finite-dimensional simple $\cL$-modules are evaluation modules enumerated by finitely supported $\Gamma$-invariant maps $\Max(S)\rightarrow P_+$.  Indeed, the only place where the Galois cocycle (and hence the isomorphism class) of the $S/R$-form plays an explicit role is in the isomorphism criterion for $\cL$-modules (Proposition \ref{mainprop}).  But in many interesting examples, even this condition vanishes, as we illustrate in Section \ref{ot}.

\bigskip

\noindent {\bf Acknowledgements:} We would like to thank Jean Auger and Zhihua Chang for their careful reading of the manuscript.

\bigskip

\noindent {\bf Notation:} Throughout this paper, $k$ will denote a field. We let $k^\times=k\setminus\{0\}$ and denote the set of nonnegative integers by $\Z_+$ .  The category of finitely generated unital commutative associative $k$-algebras will be denoted by $\kalg$, and we will write $\Max(S)$ for the maximal spectrum of each $S\in\kalg$.  

\section{Twisted forms and their maximal ideals}\label{ketto}
In this section, $k$ will denote an arbitrary field and $S/R$ will be a finite Galois extension in $\kalg$ with Galois group $\Gamma.$  Let $\fg$ be a finite-dimensional central simple  algebra over $k$, and let $R\in\kalg$.   We may view $\fg(R)\cong \fg\ot_k R$ as an algebra over $R$ by base change, the multiplication given by  $(x\ot r)(y\ot s)=xy\ot rs$ (for each $x,y\in\fg$ and $r,s\in R$).  As before,  $\cL$ will denote an $S/R$-form  of $\fg(R).$  Any such $\cL$ is obviously an algebra over $k$ by restriction of scalars, and we may thus speak of {\em $k$-ideals} and {\em $R$-ideals} of $\cL$, namely the ideals of $\cL$ viewed as an algebra over $k$ and over $R$, respectively.\footnote{We remind the reader that the word {\em ideal} means two-sided ideal.}  The goal of this section is to classify the maximal $k$-ideals of $\cL$.

\medskip

Since Galois extensions are faithfully flat, we have the following general facts.  See \cite[Thm 7.5]{matsumura}, for instance.

\begin{lemma}\label{basiclemma} Let $I$ be an ideal of $R$, and let $M$ be an $R$-module.
\begin{enumerate}
\item[{\rm(1)}] The canonical map
\begin{align*}
M&\rightarrow M\ot_R S\\
x&\mapsto x\ot 1
\end{align*}
is injective.  In particular, $R$ can be identified with a $k$-subalgebra of $S$.
\item[{\rm(2)}] After viewing $R$ inside of $S$ via (1), $IS$ is an ideal of $S$ and $R\cap IS=I$.
\end{enumerate}
\end{lemma}\qed

Up to coboundary, we can associate a Galois $1$-cocycle
$$u=(u_\gamma)_{\gamma\in\Gamma}\in Z^1\Big(\Gamma,\Aut_{S\hbox{-}alg}\big(\fg(S)\big)\Big)$$
to $\cL$, such that
$$\cL\simeq\cL_u=\{z\in\fg\ot S\ |\ u_\gamma{}^\gamma z=z\hbox{\ for all\ }\gamma\in\Gamma\}.$$
We therefore can (and henceforth will) view $\cL$ as an $R$-subalgebra of $\fg(S)=\fg\ot S$.    Note that the $S$-algebra isomorphism
$$\cL\ot_R S\simeq\fg(R)\ot_RS=\fg(S)$$
may be realized as the multiplication map
\begin{eqnarray}
\mu:&\ \cL\ot_RS&\longrightarrow\ \ \ \ \ \fg(S)\nonumber\\
&\left(\sum_i x_i\ot s_i\right)\ot s&\longmapsto\ \ \sum_ix_i\ot s_is\label{mult}
\end{eqnarray}
for all $\sum_i x_i\ot s_i\in\cL$ and $s\in S$.  This will allow us to associate an ideal of $S$ to every $R$-ideal of $\cL$.
\begin{lemma}\label{existenceofJ} Let $\cI$ be an $R$-ideal of $\cL$.  Then $\cI\ot_RS$ is an $S$-ideal of $\cL\ot_RS$, and there is a unique ideal $J=J(\cI)\subseteq S$ such that $\fg\ot_k J=\mu(\cI\ot_RS)$.
\end{lemma}
\proof Fix a $k$-basis $\{x_1,\ldots,x_m\}$ of $\fg$.  Let $J=J(\cI)$ be the set of all $s\in S$ for which there exists $\sum_{i=1}^mx_i\ot s_i\in\mu(\cI\ot_RS)$ such that $s=s_i$ for some $i$.  By the definition of $J$, it is clear that $\mu(\cI\ot_RS)\subseteq \fg\ot_k J$.  Moreover, since $\fg\ot 1\subseteq\fg\ot_k S$ is a finite-dimensional central simple $k$-algebra, it follows from the Jacobson Density Theorem that $x_i\ot s\in\mu(\cI\ot_RS)$ for all $s\in J$ and for all $i\leq m$.  Thus $\fg\ot_k J \subseteq \mu(\cI\ot_RS)$.  The uniqueness of $J$ is clear since the tensor product $\fg\ot_k J$ is being taken over a field $k$.\qed

\bigskip

\begin{proposition}\label{inclusionprop}
Let $\cI_1$ and $\cI_2$ be $R$-ideals of $\cL$.  Then $J(\cI_1)\subseteq J(\cI_2)$ if and only if $\cI_1\subseteq\cI_2$.  In particular, the map
$J:\ \{R\hbox{-ideals of\ }\cL\}\rightarrow\{\hbox{ideals of\ }S\}$
is injective.
\end{proposition}
\proof Let $\cI=\cI_1+\cI_2$.  The restriction of the multiplication map
$$\mu:\ \cL\ot_RS\rightarrow\fg(S)$$
to $\cI\ot_RS$ gives an isomorphism
$$\mu_\cI:\ \cI\ot_RS\rightarrow\fg\ot_k J(\cI)$$
with $J(\cI)=J(\cI_1)+J(\cI_2)$.  By flatness of $S/R$,
$$(\cI/\cI_2)\ot_RS\simeq\frac{\cI\ot_RS}{\cI_2\ot_RS}$$
as $S$-modules.  The injection $\mu_\cI$ restricts to an isomorphism
$$\cI_2\ot_RS\rightarrow\fg\ot_k J(\cI_2),$$
so we see that
$$\frac{\cI\ot_RS}{\cI_2\ot_RS}\simeq\frac{\fg\ot_k J(\cI)}{\fg\ot_k J(\cI_2)}=\fg\ot_k (J(\cI)/J(\cI_2)).$$
Thus $(\cI/\cI_2)\ot_RS=0$ if and only if $\fg\ot_k (J(\cI)/J(\cI_2))=0$; then by faithful flatness, $\cI/\cI_2=0$ if and only if $J(\cI)/J(\cI_2)=0$.  That is, $\cI_1\subseteq\cI_2$ if and only if $J(\cI_1)\subseteq J(\cI_2)$.\qed

\bigskip

\begin{proposition}\label{Gammastable} Let $\cI\subseteq\cL$ be an $R$-ideal.  Then $J(\cI)$ is stable under the action of the Galois group $\Gamma=\Gal(S/R)$.
\end{proposition}
\proof As in the proof of Lemma \ref{existenceofJ}, we fix a $k$-basis $\beta=\{x_1,\ldots,x_m\}$ of $\fg$.  From the definition of $J=J(\cI)$, it is easy to see that $J$ is the ideal of $S$ generated by the set $E_\beta(\cI)$ of those elements $s\in S$ for which there is an element $\sum_i x_i\ot s_i\in\cI$ for which $s_i=s$ for some $i$.  It is thus enough to show $^\gamma s\in J$ for all $\gamma\in\Gamma$ and $s\in E_\beta(\cI)$.

Let $u\in Z^1(\Gamma,\Aut_{S\hbox{-}alg}(\fg(S)))$ be a cocycle corresponding to the $S/R$-form $\cL$.  Fix $\gamma\in\Gamma$, and write $u_\gamma(x_i\ot 1)=\sum_{j=1}^m x_j\ot a_{ij}$.  Since $u_\gamma$ is an automorphism of $\fg(S)$, the matrix $A=(a_{ij})$ is invertible in $M_m(S)$.  Let $z=\sum x_i\ot s_i\in\cI$.  It suffices to show that $^\gamma s_i\in J$ for $i=1,\ldots,m$.  We have
\begin{align*}
\sum x_i\ot s_i&=\mu(z\ot 1)\\
&=\mu(u_\gamma{}^\gamma z\ot 1)\\
&=\mu\left(\sum_i u_\gamma(x_i\ot{}^\gamma s_i)\ot 1\right)\\
&=\mu\left(\sum_i\, ^\gamma s_i u_\gamma(x_i\ot 1)\ot 1\right)\\
&=\mu\left(\sum_i u_\gamma(x_i\ot 1)\ot  {^\gamma} s_i\right)\\
&=\mu\left(\sum_{i,j} x_j\ot a_{ij}\ot {^\gamma} s_i\right)\\
&=\sum_jx_j\ot\left(\sum_i a_{ij}{}^\gamma s_i\right).
\end{align*}
In matrix form, we see that
$$\left(\begin{array}{c}
^\gamma s_1\\
\vdots\\
^\gamma s_m
\end{array}\right)=(A^t)^{-1}\left(\begin{array}{c}
s_1\\
\vdots\\
s_m
\end{array}\right).$$
By definition, $s_i\in E_\beta(\cI)\subseteq J$ for all $i$, and $(A^t)^{-1}\in M_m(S)$.  Hence $^\gamma s_i\in J$ for all $i$.\qed

\bigskip

\begin{lemma}\label{correspidealofIL}
Let $I$ be an ideal of $R$.  Then $I\cL$ is an ideal of $\cL$, and $J(I\cL)=IS$.
\end{lemma}
\proof It is obvious that $I\cL$ is an ideal of $\cL$.  As $S$-modules (in fact, as $S$-algebras),
\begin{align*}
I\cL\ot_RS&=\cL\ot_RIS\\
&\simeq \cL\ot_RS\ot_SIS\\
&\simeq\fg\ot_k S\ot_SIS\\
&\simeq\fg\ot_k IS,
\end{align*}
so $J(I\cL)=IS$.\qed

\bigskip

We now turn to the classification of maximal $k$-ideals $\cI$ of the $S/R$-form $\cL$.

\begin{lemma}\label{maximal}
The sets of maximal $k$-ideals and maximal $R$-ideals of $\cL$ coincide.
\end{lemma}
\begin{proof} Let $\cI$ be a maximal $k$-ideal of $\cL.$ We claim that $\cI$ is stable under the action of $R.$ For any $r\in R$, the space $r\cI$ is clearly a $k$-ideal of $\cL$, and if $r\cI\not\subseteq\cI$, then $\cI+r\cI=\cL$ by the maximality of $\cI$.  The algebra $\cL$ is perfect by descent considerations, as has already been noted in \cite{GP}, for instance.  Thus
\begin{align*}
\cL=\cL\cL&=(\cI+r\cI)\cL\\
&=\cI\cL+\cI  (r\cL)\\
&\subseteq \cI\cL\\
&\subseteq \cI,
\end{align*}
since  $\cL$ is an $R$-algebra.  But this contradicts the proper inclusion $\cI\subsetneq\cL$, so $r\cI\subseteq \cI$ as claimed.  From this, it follows that every maximal $k$-ideal of $\cL$ is also a maximal $R$-ideal of $\cL$ and conversely. \qed
\end{proof}

\begin{lemma}\label{MSintersectionlem}
Let $M$ be a maximal ideal of $R$.  Then:
\begin{enumerate}
\item[{\rm(1)}]  There exist prime ideals of $S$ lying over $M,$ and any such ideal is maximal. The group $\Gamma$ acts transitively on the set of such maximal ideals. In particular, this set is finite.

\item[{\rm (2)}]
$$MS=\bigcap_i M_i$$
where the intersection is taken over the (finite) set of maximal ideals of $S$ lying over $M.$
\end{enumerate}

\end{lemma}
\proof (1) This is well known, but we recall the main ideas for completeness. From basic properties of Galois extensions, we know that  $R = S^\Gamma,$ and hence $S/R$ is integral. From this it follows that the set of prime ideals of $S$ lying over $M$ is not empty, that any such ideal is maximal, and that the action of $\Gamma$ on this set is transitive.  (See \cite[\S2.1 proposition 1 and \S2.2 th\'eor\`eme 2]{CA}.)

(2) Any maximal ideal $\mathfrak{m}$ of $S$ containing $MS$ will lie over $M$, since the intersection $\mathfrak{m}\cap R$ is a proper ideal of $R$ containing $MS\cap R$, which is equal to the maximal ideal $M$ by Lemma \ref{basiclemma}(2).  Thus $\mathfrak{m}=M_i$ for some $i$, and $\bigcap_i M_i$ is the radical of $MS$.  By standard base change arguments, 
$$S/MS\simeq (R/M)\ot_R S.$$
(See \cite[XVI, \S2, Proposition 7]{lang}, for instance.)

Let $L = R/M,$  a field extension of $k.$ Since the extension $S$ is Galois over $R$, general facts about base change guarantee that the extension $(R/M)\ot_R S$ is Galois over $(R/M)\ot_R R \simeq L$.  (See \cite[\S I.5]{milne}, for instance.)  That is, $S/MS$ is a Galois extension of $L$.  Galois extensions are finite \'etale and the only such extensions of $L$ are products $L_1\times\cdots\times L_m$ where the $L_i$ are finite separable field extensions of $L.$ We see from this that $S/MS$ has trivial Jacobson radical.  Hence $MS$ is a radical ideal of $S$, and  $MS=\bigcap_i M_i$.\qed

\bigskip

\begin{theorem}\label{bijection}
The map $\psi:\ I\longmapsto I\cL$ defines a bijection between the set of maximal ideals of $R$ and the set of maximal ideals of $\cL$.
\end{theorem}
\proof Let $\cI$ be a maximal ideal of $\cL$, and let $J=J(\cI)\subseteq S$ be the ideal corresponding to $\cI$.  Let $P\subseteq S$ be a maximal ideal containing $J$, and let $M=P\cap R$.  Since $S/R$ is integral, $M$ is a maximal ideal of $R$ \cite[\S2.1 proposition 1]{CA}.

As explained in Lemma \ref{MSintersectionlem}(1), the Galois group $\Gamma$ acts transitively on the finite set $M_1,\ldots,M_N$ of maximal ideals $S$ lying over $M$.  Since $J$ is $\Gamma$-stable (Proposition \ref{Gammastable}) and contained in a maximal ideal $P$ lying over $M$, we see that $J\subseteq\bigcap_{i=1}^NM_i$.  By Lemma \ref{MSintersectionlem}(2), $MS=\bigcap_{i=1}^NM_i$.  Hence $J\subseteq MS$.

Note that $M\cL$ is an ideal of $\cL$ whose corresponding ideal is $MS$, by Lemma \ref{correspidealofIL}.  By Proposition \ref{inclusionprop}, $\cI\subseteq M\cL$.  Since $MS=\bigcap_{i=1}^NM_i$ is a proper ideal of $S$, Lemma \ref{existenceofJ} guarantees that $M\cL$ is a proper ideal of $\cL$.  Hence $\cI=M\cL$ by the maximality of $\cI$, so the image of the map $\psi$ includes all maximal ideals of $\cL$.

Let $I_1$ and $I_2$ be maximal ideals of $R.$  If $I_1\cL=I_2\cL$ then $I_1S = I_2S$ by Proposition \ref{inclusionprop} and Lemma \ref{correspidealofIL}. Now Lemma \ref{basiclemma}(2) yields that $I_1 = I_2$, hence that $\psi$ is injective.  It remains only to check that $I\cL\subseteq\cL$ is maximal whenever $I\subseteq R$ is maximal.  Suppose that $I\subseteq R$ is a maximal ideal, and let $\cI\subseteq \cL$ be a maximal ideal containing $I\cL$.  We have already shown that there is a maximal ideal $M\subseteq R$ for which $\cI=M\cL$.  By Lemma \ref{basiclemma}(2) and Lemma \ref{correspidealofIL}, $M=MS\cap R=J(M\cL)\cap R=J(\cI)\cap R.$  By Proposition \ref{inclusionprop}, $J(I\cL)\subseteq J(\cI)$, so
$$I=IS\cap R=J(I\cL)\cap R\subseteq J(\cI)\cap R=M.$$
By the maximality of $I$, we see that $I=M$.  Hence $I\cL=M\cL=\cI$ is a maximal ideal of $\cL$.\qed

As an application, we recover the following well-known fact ({\it cf.}  \cite[III Cor. 5.2]{KnusOjanguren}).

\begin{corollary} Let $\mathcal A$ be an Azumaya algebra over $R.$ Every (two-sided) maximal ideal of $\mathcal A$ is of the form $I\mathcal A$ for some maximal ideal $I$ of $R.$ \qed
\end{corollary}

\section{Classification of simple modules}\label{harom}

We maintain the notation of the previous section but now assume that $\fg$ is a finite-dimensional simple Lie algebra over an algebraically closed field $k$ of characteristic zero.  The base ring $R$ will be of finite type in $\kalg$, and all modules (representations) will be of finite dimension over $k$.  Unless explicitly indicated otherwise, $\ot$ will denote a tensor product $\ot_k$ taken over the base field $k$.

Let $\cL \subseteq \fg \otimes S$ be an $S/R$-form of $\fg(R)$ as before, and let $\phi:\ \cL\rightarrow \End_k(V)$ be a finite-dimensional irreducible representation of $\cL$.  We fix a cocycle $u\in Z^1\big(\Gamma,\Aut_{S-Lie}(\fg(S))\big)$ so that $\cL=\cL_u.$

 \subsection{Evaluation maps and simple modules} Since $\cL$ is perfect, $\cL/\ker\phi$ is a finite-dimensional semisimple Lie algebra over $k$ \cite[Prop 2.1]{multiloop}.  Hence there is an isomorphism
$$f:\ \cL/\ker\phi\longrightarrow\fg_1\oplus\cdots\oplus\fg_n$$
for some collection of finite-dimensional simple $k$-Lie algebras $\fg_1,\ldots,\fg_n$.  Let $\pi:\cL\rightarrow\cL/\ker\phi$ be the natural projection.  Then
$$\cL/\ker\phi\simeq\cL/\cM_1\oplus\cdots\oplus\cL/\cM_n$$
where $\cM_1,\ldots,\cM_n$ are pairwise distinct maximal ideals of $\cL$ whose intersection is $\ker\phi$.  More precisely, we can take
$$\cM_i=\pi^{-1}\circ f^{-1}(\fg_1\oplus\cdots\oplus\widehat{\fg_i}\oplus\cdots\oplus\fg_n)$$
for $i=1,\ldots,n,$ where $\widehat{\fg_i}$ indicates that the $i$th summand is omitted.  To classify the simple modules of $\cL$, it thus suffices to consider quotients of $\cL$ by maximal ideals.\footnote{Recall that there is no difference in the concept of maximal ideal if we view $\cL$ as an $R$- or $k$-Lie algebra.}

Let $\cI\subseteq \cL$ be a maximal ideal.  By Theorem \ref{bijection}, $\cI=I\cL$ for some maximal ideal $I\subseteq R$.  Let $P\subseteq S$ be a maximal ideal lying over $I$, and let
\begin{equation}\label{eps}
\epsilon:\ S\longrightarrow S/P\simeq k
\end{equation}
be the natural evaluation map.\footnote{$S$ is of finite type over $R$ and $R$ is assumed to be of finite type over $k.$ Thus $S$ is of finite type over $k$ and therefore $S/P \simeq k$ by the Nullstellensatz.}  Then the composition
\begin{equation}\label{evmap}
\ev_P:\ \cL\hookrightarrow\fg\ot S\stackrel{1\ot\epsilon}\longrightarrow\fg\ot k\simeq \fg
\end{equation}
is a homomorphism of $k$-Lie algebras.
\begin{proposition}
The map $\ev_P:\ \cL\rightarrow\fg$ is surjective and has kernel $\cI=(P\cap R)\cL$.
\end{proposition}
\proof The multiplication map
$$\mu:\ \cL\ot_R S\longrightarrow\fg(S)$$
is an isomorphism (\ref{mult}), so given any element $x\in \fg$, there exist elements $z_i\in\cL$ and $t_i\in S$ such that
$$\mu\left(\sum_i z_i\ot t_i\right)=x\ot 1.$$
That is, if $z_i=\sum_j x_j\ot s_{ij}$ for some $k$-basis $\{ x_j\}$ of $\fg$ and $s_{ij}\in S$, then $\sum_{i,j}x_j\ot s_{ij}t_i=x\ot 1$.  Applying the map $1\ot \epsilon$ introduced in (\ref{eps}), we get $\sum_{i,j}x_j\ot \eps(s_{ij})\eps(t_i)=x\ot 1$.  But $\cL$ is closed under multiplication by elements of $k$, so $\sum_i\eps(t_i)z_i\in\cL$, and
$$\ev_P\left(\sum_i\eps(t_i)z_i\right)=\sum_{i,j}x_j\eps(s_{ij})\eps(t_i)=x.$$
Hence $\ev_P$ is surjective.

Let $z=\sum_i x_i\ot s_i\in\cL$ and $r\in I$.  Then $\eps(r)=0$, since $I=P\cap R\subseteq P=\ker\eps$.  Hence
\begin{align*}
\ev_P(rz)&=\sum x_i\eps(rs_i)\\
&=\sum x_i\eps(r)\eps(s_i)\\
&=0,
\end{align*}
so $I\cL\subseteq\ker\ev_P$.  Since $\cI=I\cL$ is a maximal ideal and $\ev_P$ is nonzero, the kernel of $\ev_P$ is precisely $\cI$.\qed

\bigskip

We have now shown that $\cL/\ker\phi$ is isomorphic to a direct sum of finitely many copies of $\fg$.  Explicitly, $\ker\phi$ is the intersection of a (finite) family of distinct maximal ideals $\cM_1,\ldots,\cM_n$ in $\cL$.  Let $I_1,\ldots,I_n$ be the (distinct) maximal ideals of $R$ given by Theorem \ref{bijection}.  For any collection $\underline{M}$ of maximal ideals $M_1,\ldots,M_n$ of $S$ lying over $I_1,\ldots ,I_n$, respectively, the map
\begin{align*}
\ev_{\underline{M}}=(\ev_{M_1},\ldots,\ev_{M_n}):\ \cL&\longrightarrow\fg\oplus\cdots\oplus\fg\\
z&\longmapsto(\ev_{M_1}(z),\ldots,\ev_{M_n}(z))
\end{align*}
descends to an isomorphism $\ev_{\underline{M}}:\cL/\ker\phi\rightarrow\fg\oplus\cdots\oplus\fg.$

Since the irreducible representations of $\fg^{\oplus n}=\fg\oplus\cdots\oplus\fg$ are precisely the tensor products
\begin{eqnarray*}
\rho=(\rho_1,\ldots,\rho_n):\fg\oplus\cdots\oplus\fg&\longrightarrow&\End_k(V_1\ot\cdots\ot V_n)\\
(x_1,\ldots,x_n)&\longmapsto&\sum_{i=1}^n\hbox{id}\ot\cdots\ot\rho_i(x_i)\ot\cdots\ot\hbox{id}
\end{eqnarray*}
of simple $\fg$-modules $(\rho_i,V_i)$, we now have a complete list of the simple $\cL$-modules.
\begin{theorem}\label{completelist}
Let $\phi:\ \cL\rightarrow\hbox{\em End}_k(V)$ be a finite-dimensional irreducible representation of $\cL$.  Then there exists a finite collection $\underline{P}=(P_1,\ldots,P_n)$ of maximal ideals of $S$ with $P_i\cap R\neq P_j\cap R$ for $i\neq j$, and a simple $\fg^{\oplus n}$-module $(\rho,V_1\ot\cdots\ot V_n)$ such that
$$V\simeq V_1\ot\cdots\ot V_n\hbox{\ and\ }\phi=\rho\circ \hbox{{\em ev}}_{\underline{P}}.$$ \qed
\end{theorem}

\bigskip

\begin{remark}{\em
The converse of Theorem \ref{completelist} is obvious.  Given a collection $P_1,\ldots,P_n$ of maximal ideals of $S$ for which the ideals $P_i\cap R$ of $R$ are pairwise distinct, the Chinese Remainder Theorem gives an isomorphism
$$\cL/\cM_1\oplus\cdots\oplus\cL/\cM_n\simeq\cL\left/\cap_i\cM_i\right.,$$
where $\cM_i=(P_i\cap R)\cL$.  (This uses the fact that the $P_i\cap R$ are maximal, as shown in the proof of Theorem \ref{bijection}.)  Thus the map
$$\cL\longrightarrow\cL/\cM_1\oplus\cdots\oplus\cL/\cM_n\simeq \fg^{\oplus n}$$
is surjective, so the pullback of any simple $\fg^{\oplus n}$-module $V=V_1\ot\cdots\ot V_n$ will be a simple $\cL$-module.
}\end{remark}

\subsection{Isomorphism classes of simple modules}\label{negy}

 Fix a Cartan subalgebra $\fh$ of $\fg$ and  an \'epinglage of $(\fg, \fh)$ (see \cite[VIII, \S4.1]{Lie}).  Given a maximal ideal $M\in\Max (S)$ and a finite dimensional representation $\rho:\ \fg\rightarrow\hbox{End}_k(W)$, we write $W(M)$ for the vector space $W$, viewed as an $\cL$-module with action given by the composition of maps
$$\cL\hookrightarrow\fg\ot S\stackrel{\ev_M}\longrightarrow\fg\stackrel{\rho}\longrightarrow\hbox{End}_k(W),$$
where $\ev_M$ is the quotient map
\begin{eqnarray*}
\ev_M:\ \fg\ot S&\longrightarrow&(\fg\ot S)/(\fg\ot M)=\fg\ot (S/M)\simeq\fg\\
x\ot s&\longmapsto&(x\ot s)(M)=s(M)x
\end{eqnarray*}
for all $x\in\fg$ and $s\in S$.  For each automorphism $\ga\in\Aut_{S-Lie}\big(\fg(S)\big)$ and $M\in\Max(S)$, we write $\ga(M)\in\Aut(\fg)$ for the automorphism defined by
$$\big(\ga(M)\big)(x)=\big(\ga(x\ot 1)\big)(M)=\ev_M\big(\ga(x\ot 1)\big),$$
for each $x\in\fg$.  It is straightforward to verify that the map
\begin{eqnarray*}
\Aut_{S-Lie}\big(\fg(S)\big)&\longrightarrow&\Aut(\fg)\\
\ga&\longmapsto&\ga(M)
\end{eqnarray*}
is a group homomorphism for each $M\in\Max(S)$.  We write $\hbox{Out}\,\ga(M)$ and $\hbox{Int}\,\ga(M)$, respectively, for the outer and inner parts, respectively, of the automorphism $\ga(M)=\hbox{Int}\,\ga(M)\circ\hbox{Out}\,\ga(M)$.  See \cite[VIII, \S5.3 Corollaire 1]{Lie} for details.

By Theorem \ref{completelist}, the (finite-dimensional) simple $\cL$-modules are those of the form $V(\gl,M)=V_{\gl_1}(M_1)\ot\cdots\ot V_{\gl_n}(M_n)$, where each $\gl_i$ is in the set $P_+^\times$ of nonzero dominant integral weights, $V_{\gl_i}$ is the simple $\fg$-module of highest weight $\gl_i$, and $M=(M_1,\ldots,M_n)$ is an $n$-tuple of maximal ideals of $S$ lying over distinct (closed) points of $\hbox{Spec}(R)$.  

\begin{lemma}\label{oversamepoints}
Suppose that the $\cL$-modules $V(\gl,M)=V_{\gl_1}(M_1)\ot\cdots\ot V_{\gl_m}(M_m)$ and $V(\mu,N)=V_{\mu_1}(N_1)\ot\cdots\ot V_{\mu_n}(N_n)$ are isomorphic for certain $\gl_1,\ldots,\gl_m,$ $\mu_1,\ldots,\mu_n\in P_+^\times$ and $M_1,\ldots,M_m,N_1,\ldots,N_n\in \Max(S)$.  Then $m=n$, and up to reordering, $M_i\cap R=N_i\cap R$ for all $i$.
\end{lemma}
\proof Let $\phi_{\gl,M}:\ \cL\longrightarrow\hbox{End}_k\big(V(\gl,M)\big)$ and $\phi_{\mu,N}:\ \cL\longrightarrow\hbox{End}_k\big(V(\mu,N)\big)$ be the homomorphisms determining the module actions.  Since $V(\gl,M)\simeq V(\mu,N)$, their kernels are equal, so
$$\bigcap_{i=1}^m(M_i\cap R)\cL=\ker\phi_{\gl,M}=\ker\phi_{\mu,N}=\bigcap_{j=1}^n(N_j\cap R)\cL.$$

By Lemma \ref{basiclemma}(2) and Lemma \ref{correspidealofIL},
\begin{align*}
\bigcap_{i=1}^m(M_i\cap R)&=\left(\bigcap_{i=1}^m(M_i\cap R)S\right)\cap R\\
&= J\left(\bigcap_{i=1}^m(M_i\cap R)\cL\right)\cap R\\
&=J\left(\bigcap_{j=1}^n(N_j\cap R)\cL\right)\cap R\\
&=\bigcap_{j=1}^n(N_j\cap R).
\end{align*}
For $I\subseteq R$, let $\Var I$ be the set of $\mathfrak{m}\in\Spec R$ with $I\subseteq \mathfrak{m}$.  Then
\begin{align*}
\bigcup_{i=1}^m\{M_i\cap R\}&=\bigcup_{i=1}^m\Var(M_i\cap R)\\
&=\Var\left(\bigcap_{i=1}^m(M_i\cap R)\right)\\
&=\Var\left(\bigcap_{j=1}^n(N_j\cap R)\right)\\
&=\bigcup_{j=1}^n\{N_j\cap R\}.
\end{align*}
Thus $m=n$, and after reordering, $M_i\cap R=N_i\cap R$ for all $i$.\qed

\bigskip

Recall that $u_\gamma$ is the image of $\gamma\in\Gamma=\Gal(S/R)$ under the Galois cocycle $u:\ \Gamma\longrightarrow\Aut_{S-Lie}\big(\fg(S)\big)$.  The group $\Gamma$ acts on the set of pairs $(\gl,M)\in P_+^\times\times\Max(S)$ by $^\gamma(\mu,N)=\left(\mu\circ\hbox{Out}\,u_\gamma^{-1}(^\gamma N), {^\gamma}N\right)$.

\begin{proposition}\label{mainprop}
Suppose $V(\gl,M)=V_{\gl_1}(M_1)\ot\cdots\ot V_{\gl_n}(M_n)$ and $V(\mu,N)=V_{\mu_1}(N_1)\ot\cdots\ot V_{\mu_n}(N_n)$ are irreducible $\cL$-modules with $\gl,\mu\in (P_+^\times)^n$ and $M_i\cap R=N_i\cap R$ for all $i$.  Then $V(\gl,M)\simeq V(\mu,N)$ if and only if there exist $\gamma_1,\ldots,\gamma_n\in\Gamma$ such that
$$(\gl_i,M_i)= {^{\gamma_i}}(\mu_i,N_i)$$
for $i=1,\ldots,n$.
\end{proposition}
\proof Let $\phi_{\gl,M}:\ \cL\rightarrow\hbox{End}_k\big(V(\gl,M)\big)$ and $\phi_{\mu,N}:\ \cL\rightarrow\hbox{End}_k\big(V(\mu,N)\big)$ be the homomorphisms defining the module actions.  Since each $\gl_i$ is nonzero, the kernel of the action of $\fg^{\oplus n}$ on $V(\gl,M)$ is trivial, and the evaluation maps $\ev_{M_i}$ induce an automorphism
$$\ev_M=\ev_{M_1}\oplus\cdots\oplus\ev_{M_n}:\ \cL/\ker\phi_{\gl,M}\stackrel{\simeq}\longrightarrow\fg^{\oplus n}.$$
Similarly, $\ev_N:\ \cL/\ker\phi_{\mu,N}\longrightarrow\fg^{\oplus n}$ is a Lie algebra isomorphism.

Let $\fg=\mathfrak{n_-}\oplus\mathfrak{h}\oplus\mathfrak{n_+}$ be the triangular decomposition of $\fg$ relative to the \'epinglage of $(\fg,\mathfrak{h})$.  We pull back the corresponding triangular decomposition of $\fg^{\oplus n}$ to obtain the triangular decomposition
\begin{equation}\label{triangulardecomposition}
\cL/\ker\phi_{\gl,M}=\ev_M^{-1}(\mathfrak{n}_-^{\oplus n})\oplus\ev_M^{-1}(\mathfrak{h}^{\oplus n})\oplus\ev_M^{-1}(\mathfrak{n}_+^{\oplus n}).
\end{equation}
The representations $V(\gl,M)$ and $V(\mu,N)$ will be isomorphic precisely when they have the same highest weights relative to the decomposition (\ref{triangulardecomposition}).

 The Galois group $\Gamma=\Gal(S/R)$ acts transitively on the fibres of the pullback map $\Spec (S)\rightarrow \Spec(R)$ over maximal ideals of $R$.  Choose $\gamma_i\in\Gamma$ so that $M_i= {^{\gamma_i}}N_i$ for all $i$.

Let $\fg^i=0\oplus\cdots\oplus\fg\oplus\cdots\oplus 0$ be the $i$th component of $\fg^{\oplus n}$. Note that
\begin{align*}
\ev_M^{-1}(\fg^i)=&\bigcap_{r\neq i}\ker\ev_{M_r}\\
=&\bigcap_{r\neq i}(M_r\cap R)\cL\\
=&\bigcap_{r\neq i}(N_r\cap R)\cL\\
=&\bigcap_{r\neq i}\ker\ev_{N_r}.
\end{align*}
Therefore, $\ev_{N_j}\circ\ev_M^{-1}(\fg^i)=0$ for all $i\neq j$, and
\begin{equation*}
\ev_N\circ\ev_M^{-1}(x^i)=\iota_i\circ\ev_{N_i}\circ\ev_M^{-1}(x^i)=\iota_i\circ\ev_{N_i}\circ\ev_{M_i}^{-1}(x),
\end{equation*}
for all $x^i\in\fg^i$, where $\iota_i$ is the inclusion of $\fg$ as the $i$th component of $\fg^{\oplus n}$:
$$\iota_i:\fg\hookrightarrow 0\oplus\cdots\oplus\fg\oplus\cdots\oplus 0\subseteq\fg^{\oplus n}.$$

Relative to the decomposition (\ref{triangulardecomposition}), the highest weight of $V(\gl,M)$ is thus $\displaystyle{\sum_{i=1}^n\gl_i\circ\ev_{M_i}}$ and the highest weight of $V(\mu,N)$ is $\displaystyle{\sum_{i=1}^n\nu_i\circ\ev_{N_i}}$, where $\nu_i\in\big(\ev_{N_i}\circ\ev_{M_i}^{-1}(\fh)\big)^*$ is the highest weight of $V_{\mu_i}$, relative to the new triangular decomposition
$$\fg=\ev_{N_i}\circ\ev_{M_i}^{-1}(\mathfrak{n}_-)\oplus\ev_{N_i}\circ\ev_{M_i}^{-1}(\mathfrak{h})\oplus\ev_{N_i}\circ\ev_{M_i}^{-1}(\mathfrak{n}_+).$$

By \cite[Lemma 5.2]{multiloop}, $\nu_i=\mu_i\circ\tau_i^{-1}$, where $\tau_i=\hbox{Int}(\ev_{N_i}\circ\ev_{M_i}^{-1}).$  That is, $V(\gl,M)\simeq V(\mu,N)$ if and only if
$$\sum_{i=1}^n\gl_i\circ\ev_{M_i}= \sum_{i=1}^n\mu_i\circ\tau_i^{-1}\circ\ev_{N_i}$$
on $\ev_M^{-1}(\fh^{\oplus n})$.  For the $i$th component $\fh^i=0\oplus\cdots\oplus\fh\oplus\cdots\oplus 0$, we have $\ev_M^{-1}(\fh^i)\subseteq\ev_M^{-1}(\fg^i)=\bigcap_{j\neq i}(M_j\cap R)\cL$, so $\gl_j\circ\ev_{M_j}(\ev_M^{-1}(\fh^i))=0$ for $i\neq j$.  Therefore, $V(\gl,M)\simeq V(\mu,N)$ if and only if $\gl_i\circ\ev_{M_i}=\mu_i\circ\tau_i^{-1}\circ\ev_{N_i}$ for all $i$; that is, if and only if $\lambda_i=\mu_i\circ\hbox{Out}(\ev_{N_i}\circ\ev_{M_i}^{-1}).$

We now simplify the expression for the automorphism $\ev_{N_i}\circ \, \ev_{M_i}^{-1}:\fg\rightarrow\fg$.  For $x\in\fg$, write $\ev_{M_i}^{-1}(x)=\sum_jx_j\ot s_j+\ker{\ev_{M_i}}\in\cL/\ker\ev_{M_i}=\cL/\ker\ev_{N_i}$, where $x_j\in\fg$ and $s_j\in S$ for all $j$.  Then $\ev_{N_i}\circ\ev_{M_i}^{-1}(x)=\sum_js_j(N_i)x_j.$  By definition,
$$s_j(N_i)+N_i=s_j+N_i\in S/N_i,$$
and $s_j(N_i)\in k\subseteq R$ is clearly fixed by $\gamma_i\in\Gamma$.  Hence
$$s_j(N_i)+ {^{\gamma_i}}N_i = {^{\gamma_i}}s_j+ {^{\gamma_i}}N_i\in S/ {^{\gamma_i}}N_i=S/M_i,$$
and $s_j(N_i)= {^{\gamma_i}}s_j(M_i).$  Therefore,
$$\ev_{N_i}\circ\ev_{M_i}^{-1}(x)=\sum_j\, {^{\gamma_i}}s_j(M_i)x_j.$$
Moreover, $\sum_j x_j\ot s_j\in\cL=\{z\in\fg\ot S\ |\ u_\gamma \ ^\gamma z=z\hbox{\ for all\ }\gamma\in\Gamma\}$, so
\begin{align*}
\ev_{N_i}\circ\ev_{M_i}^{-1}(x)=& {^{\gamma_i}}\left(\sum_jx_j\ot s_j\right)(M_i)\\
=&(u_{\gamma_i})^{-1}\left(\sum_jx_j\ot s_j\right)(M_i)\\
=&u_{\gamma_i}^{-1}(M_i)\sum_js_j(M_i)x_j\\
=&u_{\gamma_i}^{-1}(M_i)(x),
\end{align*}
 and $\ev_{N_i}\circ\ev_{M_i}^{-1}=u_{\gamma_i}^{-1}(M_i)$.  Hence $V(\gl,M)\simeq V(\mu,N)$ if and only if there exist $\gamma_1,\ldots,\gamma_n\in\Gamma$ such that $^{\gamma_i}(\mu_i,N_i)=(\gl_i,M_i)$ for all $i$.\qed

\bigskip
\bigskip

We identify the $\cL$-module $V(\gl,M)=V_{\gl_1}(M_1)\ot\cdots\ot V_{\gl_n}(M_n)$ with the map
$$\chi_{[\gl,M]}:\ \hbox{Max}(S)\rightarrow P_+,$$
where $\chi_{[\gl,M]}=\sum_{\gamma\in\Gamma}\sum_{i=1}^n\chi_{^\gamma(\gl_i,M_i)}$ and
\begin{eqnarray*}
\chi_{(\mu_i,N_i)}:\ \hbox{Max}(S)&\rightarrow&P_+\\
I&\mapsto&\left\{\begin{array}{ll}
\mu_i &\hbox{if}\ I=N_i\\
0 &\hbox{otherwise.}
\end{array}\right.
\end{eqnarray*}
The Galois group $\Gamma$ acts on the set $\mathcal{F}$ of finitely supported functions $f:\ \hbox{Max}(S)\rightarrow P_+$, by identifying each function $f$ with the set of ordered pairs $\{(f(M),M)\ |\ M\in\hbox{Max}(S)\}$ and defining ${^\gamma}f=\{{^\gamma}\big(f(M),M\big)\ |\ M\in\hbox{Max}(S)\}$.  The function $\chi_{[\gamma,M]}$ is then $\Gamma$-invariant, and the set $\mathcal{F}^\Gamma$ of $\Gamma$-invariant functions in $\mathcal{F}$ is in bijection with the set $\mathcal{C}$ of isomorphism classes $[V]$ of (finite-dimensional) simple $\cL$-modules $V$:
\begin{theorem}\label{correspondence} The map $\psi:\ [V(\gl,M)]\mapsto \chi_{[\gl,M]}$ is a well-defined natural bijection between $\mathcal{C}$ and $\mathcal{F}^\Gamma$.
\end{theorem}
\proof By Theorem \ref{completelist}, Lemma \ref{oversamepoints}, and Proposition \ref{mainprop}, two simple $\cL$-modules $W_1$ and $W_2$ are isomorphic if and only if there exist $n\geq 0$, ordered pairs $(M,\gl),(N,\mu)\in\big(\Max(S)\big)^n\times(P_+^\times)^n$ with $M_i\cap R=N_i\cap R\neq N_j\cap R = M_j\cap R$ for $i\neq j$, and $\gamma_1,\ldots,\gamma_n\in\Gamma$ such that $W_1\simeq V(\gl,M)$, $W_2\simeq V(\mu,N)$, and $(M_i,\gl_i)={^{\gamma_i}}(N_i,\mu_i)$ for all $i$.  Thus $V(\gl,M)\simeq V(\mu,N)$ if and only if $\chi_{[\gl,M]}=\chi_{[\mu,N]}$.  In particular, the map $\psi:\ \mathcal{C}\rightarrow\mathcal{F}^\Gamma$ is well-defined and injective.  It is also surjective, as the support of any $f\in\mathcal{F}^\Gamma$ decomposes into a disjoint union of $\Gamma$-orbits.  Therefore, $\displaystyle{f=\sum_{\gamma\in\Gamma}\sum_{i=1}^m\chi_{^\gamma(\gl_i,M_i)}}$ for some collection of orbit representatives $M_1,\ldots ,M_m\in\Max(S)$.\qed

\section{Applications}\label{ot}
Throughout this section, $k$ will denote an algebraically closed field of characteristic zero.
\subsection{Multiloop algebras}\label{multiloopsection}
Multiloop algebras are multivariable generalizations of the loop algebras in affine Kac-Moody theory.  The study of these algebras and their extensions includes a substantial body of work on (twisted and untwisted) multiloop, toroidal, and extended affine Lie algebras.  The representation theory of multiloop algebras has also been adapted to include generalized current algebras and equivariant map algebras \cite{CFK,NSS}.  When $R$ and $S$ are Laurent polynomial rings, the intersection of these classes of algebras with the class of twisted forms discussed in the present paper includes multiloop algebras (\ref{multiloopsection}), but not Margaux algebras (\ref{Azumayasection}), for instance.

Let $\fg$ be a finite-dimensional simple Lie algebra over $k$, with commuting automorphisms $\gs_1,\ldots,\gs_N:\ \fg\longrightarrow\fg$ of finite orders $m_1,\ldots,m_N$, respectively.  Fix a primitive $m_j$th root of unity $\xi_j\in k$ for each $j$, and let $R=k[t_1^{\pm m_1},\ldots,t_N^{\pm m_N}]\subseteq S=k[t_1^{\pm 1},\ldots,t_N^{\pm 1}].$

The (twisted) multiloop algebra $\cL=\cL(\fg,\gs)$ is a $\Z^N$-graded subalgebra of $\fg(S)=\fg\ot S$:
$$\cL(\fg,\gs)=\bigoplus_{j\in \Z^N}\fg_j\ot t^j,$$
where $j=(j_1,\ldots,j_N)$, $\fg_j=\{x\in\fg\ |\ \gs_i(x)=\xi_i^{j_i}x\hbox{\ for\ }i=1,\ldots,N\}$, and $t^j=t_1^{j_1}t_2^{j_2}\cdots t_N^{j_N}$.  It is easy to see that $\cL$ is a Lie algebra over $R$ and an $S/R$-form of $\fg(R)$.  

Specializing our main theorems to the case of multiloop algebras, we recover the results of \cite{multiloop}.  Maximal ideals $M_i=M_{a_i}=(t_1-a_{i1},\ldots,t_N-a_{iN})$ of $S$ correspond to points $a_i=(a_{i1},\ldots,a_{iN})$ on the algebraic $n$-torus $(k^\times)^N=k^\times\times\cdots\times k^\times$.  Note that $M_i\cap R$ is the ideal (of $R$) of polynomials vanishing at $a_i$.  Thus $M_i\cap R\in \hbox{Max}\,R$ is generated by $\{t_1^{m_1}-a_{i1}^{m_1},\ldots,t_N^{m_N}-a_{iN}^{m_N}\}$.  Therefore, $M_i\cap R=M_j\cap R$ if and only if $m(a_i)=m(a_j)$, where we write $m(a_\ell)=(a_{\ell 1}^{m_1},\ldots,a_{\ell N}^{m_N})$ for all $a_\ell\in(k^\times)^N$.

The Galois group $\Gamma=\Gal(S/R)$ is $\Z_{m_1}\times\cdots\times\Z_{m_N},$ where each $\Z_{m_i}$ is generated by an element 
$\displaystyle{\ga_i:\ t_j\mapsto\left\{\begin{array}{ll}
\xi_it_i&\hbox{if\ }i=j\\
t_j&\hbox{otherwise.}
\end{array}\right.}$
The $1$-cocycle $u:\ \Gamma\rightarrow\Aut_{S-Lie}(\fg(S))$ corresponding to $\cL$ is given by
$$u_\gamma=\gs_1^{-r_1}\cdots\gs_N^{-r_N}\ot 1,$$
for each $\gamma=(\ga_1^{r_1},\ldots,\ga_N^{r_N})\in\Gamma.$
Then $u_\gamma(M)=\gs_1^{-r_1}\cdots\gs_N^{-r_N}$ for all $M\in\Max(S)$.  The fact that 
\begin{align*}
u_\gamma:\ \Max(S)&\rightarrow \Aut\,\fg   \\
M&\mapsto u_\gamma(M)
\end{align*}
is constant means that the action of $\Gamma$ on $P_+^\times\times\Max(S)$ splits into separate actions of $\Gamma$ on $\Max(S)$ and on $P_+^\times$ by
\begin{align*}
\psi:\ \Gamma\times P_+^\times&\rightarrow P_+^\times\\
(\gamma,\gl)&\mapsto \gl\circ\hbox{Out}\,\gs_1^{-r_1}\cdots\gs_N^{-r_N}.
\end{align*}
In this language, $\Gamma$ acts on $P_+^\times\times\Max(S)$ as $^\gamma(\gl,M)=(\psi(\gamma^{-1},\gl), ^\gamma M)$.  The $\Gamma$-invariant functions $\chi_{[\gl,M]}:\ \Max(S)\rightarrow P_+$ become $\Gamma$-equivariant functions under the new action $\psi$ on $P_+^\times$.  We thus recover the following theorem \cite[Cor 4.4, Thm 4.5, and Cor 5.10]{multiloop}:
\begin{theorem}

\begin{enumerate}
\item[{\rm(1)}] The finite-dimensional simple modules of $\cL(\fg;\gs)$ are those of the form $V(\gl,a)=V_{\gl_1}(M_{a_1})\ot\cdots\ot V_{\gl_n}(M_{a_n})$ for $n\geq 0$, $a_i\in(k^\times)^N$, and $m(a_i)\neq m(a_j)$ whenever $i\neq j$.  
\item[{\rm(2)}] The isomorphism classes of finite-dimensional simple $\cL(\fg;\gs)$-modules are in bijection with the finitely supported $\Gamma$-equivariant maps $(k^\times)^N\rightarrow P_+$.
\end{enumerate}
\end{theorem}

\subsection{Azumaya and Margaux algebras}\label{Azumayasection}
Fix Laurent polynomial rings $R=k[t_1^{\pm 2},t_2^{\pm 2}]$ and $S=k[t_1^{\pm 1},t_2^{\pm 1}]$.  Let $A=A(1,2)$ be the standard Azumaya algebra, the unital associative $R$-algebra generated by $\{T_1^{\pm 1},T_2^{\pm 1}\}$ with relations $T_2T_1=-T_1T_2$ and $T_i^2=t_i^2$ for $i=1,2$.  Then $A$ is an $S/R$-form of the associative algebra $M_2(R)$ of $2\times 2$ matrices over $R$, as can be readily verified using one of the well-known representations of the quaternions as matrices in $M_2(\mathbb{C})$.

Since ${\bf PGL_2}$ is the automorphism group (scheme) of both $M_2(k)$ and $\mathfrak{sl}_2(k)$, there is a natural correspondence between $S/R$-forms of $M_2(R)$ and $\mathfrak{sl}_2(R)$.  Namely, given any $S/R$-form $B$ of the matrix algebra $M_2(R)$, view $B$ as a Lie algebra $\hbox{Lie}\,B$ with bracket $[a,b]=ab-ba$.  Its derived subalgebra $(\hbox{Lie}\,B)'=\hbox{Span}\{[a,b]\ |\ a,b\in B\}$ is then an $S/R$-form of $\mathfrak{sl}_2(R)$.

Applying this construction to $\cL_1=(\hbox{Lie}\,A)'$ and computing explicitly, it follows that $\cL_1\simeq\cL(\mathfrak{sl}_2(k),\gs_1,\gs_2)$ where $\gs_1$ and $\gs_2$ are conjugation by $\displaystyle{\left(\begin{array}{cc}
1 &0\\
0&-1
\end{array}\right)}$ and $\displaystyle{\left(\begin{array}{cc}
0 &1\\
1&0
\end{array}\right)},$ respectively \cite{GP}.  Therefore, we obtain the representations of $\cL_1$ as in the previous section.

Surprisingly, not every twisted form of $\fg(k[t_1^{\pm 1},t_2^{\pm 1}])$ is a multiloop algebra.  This can be seen using loop torsors.  The only known $S/R$-forms of $\fg(R)$ which are not isomorphic to multiloop algebras are called {\em Margaux algebras}.  The simplest of these can be constructed concretely as follows.  See \cite{GP} for details.

Let $A$, $R$, and $S$ be as in Section \ref{multiloopsection}.  The right $A$-module
$$M=\{(\gl,\mu)\in A\oplus A\ |\ (1+T_1)\gl=(1+T_2)\mu\}$$
is projective but not free.  This can be used to show that its endomorphism ring $\cM=\hbox{End}_A(M)$, while also an $S/R$-form of $M_2(R)$, is not isomorphic to $A$ as an $A$-algebra.  It follows that $\cL_1$ and $\cL_2=(\hbox{Lie}\,\cM)'$ are non-isomorphic $S/R$-forms of $\mathfrak{sl}_2(R)$.  By the classification of involutions in $\hbox{PGL}_2(k)$ and a study of loop torsors, it can be shown that $\cL_2$ is {\em not} a (twisted) multiloop algebra.

By Theorems \ref{completelist} and \ref{correspondence}, the irreducible representations of $\cL_2$ are the tensor products $V(\gl,M)=V_{\gl_1}(M_1)\ot\cdots\ot V_{\gl_n}(M_n)$, where $\gl_1,\ldots,\gl_n\in\Z_+\setminus\{0\}$ are highest weights of $\mathfrak{sl}_2(k)$ and $M_i=\langle t_1-a_{i1},t_2-a_{i2}\rangle$ are maximal ideals of $S=k[t_1^{\pm 1},t_2^{\pm 1}]$ corresponding to points in distinct fibres over $\Spec R$.  That is, $(a_{i1}^2,a_{i2}^2)\neq (a_{j1}^2,a_{j2}^2)$ for $i\neq j$.

Two such representations $V(\gl,M)=V_{\gl_1}(M_1)\ot\cdots\ot V_{\gl_m}(M_m)$ and $V(\mu,N)=V_{\mu_1}(N_1)\ot\cdots\ot V_{\mu_n}(N_n)$ are isomorphic precisely when the corresponding $\Gal(S/R)$-invariant functions $\chi_{[\gl,M]}$ and $\chi_{[\mu,N]}$ are equal.  But the action
$$^\gamma(\gl_i,M_i)=(\gl_i\circ\hbox{Out}\,u_\gamma^{-1}(^\gamma M_i),{^\gamma} M_i)$$
is simply an action on $\Max(S)$:
$${^\gamma}(\gl_i,M_i)=(\gl_i, {^\gamma} M_i),$$
since $u_\gamma^{-1}(^\gamma M)\in\hbox{Aut}\,\mathfrak{sl}_2(k)$, and every automorphism of $\mathfrak{sl}_2(k)$ is inner!  Thus $V(\gl,M)\simeq V(\mu,N)$ if and only if (after reordering the tensor factors) $m=n$, $\gl_i=\mu_i$, and the $a_i,b_i\in k^\times\times k^\times$ corresponding to $M_i$ and $N_i$ satisfy $a_{ij}=\pm b_{ij}$ for all $i$ and $j$.

As for any Galois extension $S/R$, the isomorphism classes of the (finite-dimensional) simple modules of any $S/R$-form of $\mathfrak{sl}_2(R)$ are given by restrictions of the same evaluation modules of $\mathfrak{sl}_2(S)$.  In particular, the irreducible $\cL_1$- and $\cL_2$-modules come from the same $\mathfrak{sl}_2(S)$-modules.

\end{document}